\PassOptionsToPackage{backref=page}{hyperref}
\documentclass[12pt,letterpaper,reqno]{amsart}

\usepackage[T1]{fontenc}
\usepackage{amsmath,amssymb,amsfonts,amsthm}
\usepackage{mathtools}
\usepackage{aliascnt}
\usepackage{microtype}
\usepackage{enumitem}
\usepackage{booktabs}
\usepackage{xcolor}
\usepackage{doi}
\usepackage{hyperref}
\usepackage{tikz}
\usetikzlibrary{positioning,arrows.meta}
\usepackage{bookmark}
\usepackage[capitalize,noabbrev]{cleveref}

\hypersetup{
  pdftitle={Tilting Completion and Full-Rank Self-Orthogonal Modules},
  pdfauthor={Wen Chang and Quanyu Tang},
  pdfsubject={Tilting modules, tilting completion, and presilting complexes},
  pdfkeywords={tilting module, pretilting module, tilting completion, presilting complex, quasi-hereditary algebra, one-point extension},
  pdfstartview={FitH},
  colorlinks=true,
  linkcolor=blue!55!black,
  citecolor=green!35!black,
  urlcolor=blue!55!black
}
\renewcommand*{\backref}[1]{}
\renewcommand*{\backrefalt}[4]{%
  \ifcase #1
  \or
    \quad #2%
  \else
    \quad #2%
  \fi
}
\newtheorem{theorem}{Theorem}[section]

\newaliascnt{lemma}{theorem}
\newtheorem{lemma}[lemma]{Lemma}
\aliascntresetthe{lemma}

\newaliascnt{proposition}{theorem}
\newtheorem{proposition}[proposition]{Proposition}
\aliascntresetthe{proposition}

\newaliascnt{corollary}{theorem}
\newtheorem{corollary}[corollary]{Corollary}
\aliascntresetthe{corollary}

\newaliascnt{conjecture}{theorem}

\aliascntresetthe{conjecture}

\newaliascnt{problem}{theorem}
\newtheorem{problem}[problem]{Problem}
\aliascntresetthe{problem}

\theoremstyle{definition}
\newaliascnt{definition}{theorem}

\aliascntresetthe{definition}

\newaliascnt{example}{theorem}

\aliascntresetthe{example}

\newaliascnt{remark}{theorem}
\newtheorem{remark}[remark]{Remark}
\aliascntresetthe{remark}

\crefname{theorem}{Theorem}{Theorems}
\Crefname{theorem}{Theorem}{Theorems}
\crefname{lemma}{Lemma}{Lemmas}
\Crefname{lemma}{Lemma}{Lemmas}
\crefname{proposition}{Proposition}{Propositions}
\Crefname{proposition}{Proposition}{Propositions}
\crefname{corollary}{Corollary}{Corollaries}
\Crefname{corollary}{Corollary}{Corollaries}
\crefname{conjecture}{Conjecture}{Conjectures}
\Crefname{conjecture}{Conjecture}{Conjectures}

\DeclareMathOperator{\add}{add}
\DeclareMathOperator{\thick}{thick}
\DeclareMathOperator{\proj}{proj}
\DeclareMathOperator{\pd}{pd}
\DeclareMathOperator{\rad}{rad}
\DeclareMathOperator{\Hom}{Hom}
\DeclareMathOperator{\Ext}{Ext}
\newcommand{\Db}{\mathrm{D}^{\mathrm b}}
\newcommand{\Kb}{\mathrm{K}^{\mathrm b}}
\newcommand{\Cb}{\mathrm{C}^{\mathrm b}}
\newcommand{\C}{\mathbb C}

\begin{document}

\title[Tilting Completion and Self-Orthogonality Modules]
{Tilting Completion and Full-Rank Self-Orthogonal Modules}

\author[W.~Chang]{Wen Chang}
\address{School of Mathematics and Statistics, Shaanxi Normal University,
Xi'an 710062, P.~R.~China}
\email{changwen161@163.com}

\author[Q.~Tang]{Quanyu Tang}
\address{School of Mathematical Sciences, University of Science and Technology
of China, Hefei 230026, P.~R.~China}
\email{tangquanyu827@gmail.com}

\subjclass[2020]{Primary 16D90; Secondary 16E35, 16G10}
\keywords{tilting module, pretilting module, tilting completion, presilting
complex, quasi-hereditary algebra, one-point extension}
\date{}
\begin{abstract}
We give negative answers to two tilting-completion questions for
finite-dimensional algebras.  We construct two finite-dimensional basic
connected quasi-hereditary $\mathbb C$-algebras.  The first admits a faithful
basic full-rank pretilting module with no tilting completion; the second admits
an almost-tilting module with no tilting completion.
The full-rank example also yields counterexamples to two conjectures:
Enomoto's Self-orthogonal Wakamatsu-tilting Conjecture and the
Self-orthogonal Faithful Conjecture of Chen, Li, Zhang, and Zhao.
We further show that the Self-orthogonal Wakamatsu-tilting Conjecture holds
for all finite-dimensional algebras if and only if the Self-orthogonal
Faithful Conjecture holds for all finite-dimensional algebras.

The construction ultimately stems from Krah's non-full exceptional
collection of maximal length on a rational surface and Kalck's associated
full-rank presilting example.
\end{abstract}

\maketitle
\setcounter{tocdepth}{1} 

\tableofcontents
\section{Introduction}

Tilting theory was introduced and developed in the representation theory of
finite-dimensional algebras around the beginning of the 1980s, and has since
become one of the standard tools connecting module categories, homological
algebra, and derived categories.  Brenner and Butler
\cite{BrennerButler1980} initiated the theory by extending the reflection
functors of Bernstein--Gelfand--Ponomarev \cite{BernsteinGelfandPonomarev1973}
and Auslander--Platzeck--Reiten \cite{AuslanderPlatzeckReiten1979} to tilting
functors.  Happel and Ringel \cite{HappelRingel1982} gave the modern
module-theoretic form of classical tilting theory.  Miyashita
\cite{Miyashita1986} and Happel \cite{Happel1987} subsequently extended the
notion to tilting modules of arbitrary finite projective dimension, and
Happel showed that tilting modules induce derived equivalences.  Rickard
\cite{Rickard1989} then generalized the picture further by introducing
tilting complexes and characterizing derived equivalences of rings in these
terms.

In this paper, we consider tilting modules in the generalized sense of
Miyashita and Happel.  Throughout, all algebras are finite-dimensional over
the stated field, and all modules are finite-dimensional left modules.  For an algebra $A$, let $n=n(A)$ denote the number of isomorphism classes of simple
$A$-modules. For an $A$-module $X$, let $|X|$ denote the number of isomorphism
classes of indecomposable direct summands of $X$; in particular, $|A|=n(A)$. An $A$-module $T$ is called a
\emph{tilting module} if the following conditions hold:
\begin{enumerate}[label=\textnormal{(T\arabic*)},leftmargin=2.5em]
  \item $T$ has finite projective dimension, i.e., $\pd_A T<\infty$;
  \item $T$ is self-orthogonal, i.e., $\Ext_A^i(T,T)=0$ for every $i>0$;
  \item there is an exact sequence
  \[
    0\longrightarrow A\longrightarrow T_0\longrightarrow T_1
    \longrightarrow\cdots\longrightarrow T_s\longrightarrow0
  \]
  with $T_i\in\add(T)$ for every $i$.
\end{enumerate}
If only \textnormal{(T1)} and \textnormal{(T2)} are satisfied, we call $T$ a
\emph{pretilting module}.  A pretilting module is called
\emph{partial-tilting} if it is a direct summand of a tilting module.  Thus,
in our terminology, ``pretilting'' does not include the existence of a
complement.  We will always pass to basic representatives when counting
indecomposable summands.  A basic pretilting module with $n-1$
indecomposable direct summands will also be called \emph{almost-tilting}.
The terminology in the literature is not completely uniform: Rickard and
Schofield \cite{RickardSchofield1989}, for example, use ``partial generalized
tilting'' for what we call pretilting, while Coelho--Happel--Unger
\cite{CoelhoHappelUnger1994} use ``partial tilting'' in a different sense. Further details can be found in \cite{Chang2025}, whose conventions we adopt in this paper.

Every basic tilting $A$-module has exactly $n$ indecomposable direct
summands \cite[Theorem~1.19]{Miyashita1986}.  For $1\leq m\leq n$, one may therefore ask the following
completion question.

\begin{problem}[The tilting-completion question $\mathrm{(C_{m})}$]
Let $A$ be a finite-dimensional algebra and let $M$ be a basic pretilting
$A$-module with $|M|=m$.  Must $M$ be partial-tilting?
\end{problem}

The history of this problem already shows a sharp difference between
classical and generalized tilting theory.  In the classical case, where the
projective dimension is at most one, Bongartz's completion theorem
\cite{Bongartz1981} shows that every pretilting module can be completed to a
tilting module; hence $\mathrm{(C_{m})}$ has an affirmative answer for every
$m$.  For generalized tilting modules, Rickard and Schofield
\cite{RickardSchofield1989} formulated the completion problem and proved that
it has an affirmative answer for representation-finite algebras.  They also
constructed a representation-infinite algebra of rank three with a
one-summand pretilting module which has no tilting completion, showing that
the unrestricted problem is genuinely different.  Subsequent work of
Coelho, Happel and Unger \cite{CoelhoHappelUnger1994} related the existence
of complements to contravariant finiteness of suitable Ext-orthogonal
subcategories.  It is also worth noting that the situation changes if one
allows infinitely generated complements: finite-dimensional pretilting
modules admit completions to possibly infinitely generated tilting modules
under substantially more general hypotheses; see
\cite{ColpiTrlifaj1995,AngeleriHugelCoelho2002}.

More recently, the first author \cite{Chang2025} studied the questions
$\mathrm{(C_{m})}$ for gentle algebras.  By using the surface model of the module category over a gentle algebra \cite{BaurCoelho2021,Chang2026}, it is proved that for a gentle algebra of rank $n$, both
$\mathrm{(C_n)}$ and $\mathrm{(C_{n-1})}$ have affirmative answers.  On the
other hand, for every $n\geq3$ and every $1\leq m\leq n-2$, there exists a
connected gentle algebra of rank $n$ carrying a rank-$m$ pretilting module
which is not partial-tilting.

To the best of our knowledge, before the present work the two extremal questions
$\mathrm{(C_n)}$ and $\mathrm{(C_{n-1})}$ were still open in the general
finite-dimensional setting.  The first asks whether the coresolution axiom
\textnormal{(T3)} in the definition of a tilting module can be replaced by the
numerical condition $|T|=n$.  The second asks whether every almost-tilting
module admits a tilting complement.  Our first main result gives a negative
answer to the full-rank question, even inside the class of connected
quasi-hereditary algebras.  In particular, after passing to the opposite algebra to reconcile
module conventions, it provides, to the best of our knowledge, the first
counterexample to the rank-tilting-completion property.

\begin{theorem}
\label{thm:Cn}
There exist a finite-dimensional basic connected quasi-hereditary $\C$-algebra $B$
and a faithful basic pretilting $B$-module $M$ such that $|M|=n(B)$, but $M$ is not a
direct summand of any tilting $B$-module.  Consequently,
$\mathrm{(C_n)}$ has a negative answer.
\end{theorem}

The full-rank example also settles a conjecture of Enomoto on self-orthogonal
modules.  Enomoto \cite{Enomoto2023} introduced projectively Wakamatsu tilting
modules and conjectured that, for an Artin algebra $\Lambda$, every
self-orthogonal $\Lambda$-module $T$ satisfying $|T|=|\Lambda|$ is Wakamatsu
tilting \cite[Conjecture~5.9]{Enomoto2023}.  Since Enomoto works with finitely
generated right modules, we apply his statements below to opposite algebras.
Corollary~\ref{cor:enomoto-consequences} shows that the module $M$ in
Theorem~\ref{thm:Cn} is not Wakamatsu tilting.  Consequently, after viewing
$M$ as a right $B^{\mathrm{op}}$-module, Enomoto's conjecture has a negative
answer even for faithful basic modules over basic connected quasi-hereditary
algebras.  The same corollary shows that $M$ has no Bongartz completion in
Enomoto's sense and that
\[
  M^{\perp}
  =
  \{X\in B\text{-}\mathrm{mod}\mid
    \Ext_B^i(M,X)=0\text{ for every }i>0\}
\]
has no finite cover and is not covariantly finite.  

Chen, Li, Zhang, and Zhao \cite[Section~5]{ChenLiZhangZhao2025} formulated the
Self-orthogonal Faithful Conjecture, which asserts that every self-orthogonal
$A$-module $T$ satisfying $|T|=n(A)$ is faithful.  We prove a general reduction
from the failure of Wakamatsu tilting to the failure of faithfulness and apply
it to Theorem~\ref{thm:Cn}.

\begin{theorem}
\label{thm:SFC-counterexample}
There exist a finite-dimensional basic connected quasi-hereditary
$\C$-algebra $\Gamma$ and a basic $\Gamma$-module $T$ such that
\[
  |T|=n(\Gamma),
  \qquad
  \pd_\Gamma T<\infty,
  \qquad
  \Ext_\Gamma^i(T,T)=0\quad(i>0),
  \qquad
  \operatorname{Ann}_\Gamma(T)\neq0.
\]
Consequently, the Self-orthogonal Faithful Conjecture has a negative answer.  Moreover, $T$ is not a
direct summand of any tilting $\Gamma$-module.
\end{theorem}

The reduction proving Theorem~\ref{thm:SFC-counterexample} also shows that,
as universal assertions over all finite-dimensional algebras, Enomoto's
Self-orthogonal Wakamatsu-tilting Conjecture and the Self-orthogonal Faithful
Conjecture are equivalent; see
Corollary~\ref{cor:SWC-SFC-equivalent-universally}.

The adjacent almost-full-rank case also fails.

\begin{theorem}
\label{thm:Cnminusone}
There exist a finite-dimensional basic connected quasi-hereditary
$\C$-algebra $\Lambda$ and a basic pretilting $\Lambda$-module $L$ such that
$|L|=n(\Lambda)-1$, but $L$ is not a direct summand of any tilting
$\Lambda$-module. Consequently, $\mathrm{(C_{n-1})}$ has a negative answer
even for connected quasi-hereditary algebras.
\end{theorem}

Our construction starts from the derived side. A complex $P^{\bullet}\in\Kb(\proj A)$ is \emph{presilting} if
$\Hom_{\Kb(\proj A)}(P^{\bullet},P^{\bullet}[i])=0$ for every $i>0$,
and it is \emph{silting} if, in addition,
$\thick(P^{\bullet})=\Kb(\proj A)$, see \cite[Definition 2.1]{AiharaIyama2012}. Krah \cite{Krah2024} constructed a non-full
exceptional collection of maximal length on a rational surface.  Kalck
\cite{Kalck2023} observed that, after passing through a tilting equivalence,
this yields a full-rank presilting complex over a finite-dimensional algebra
which cannot be completed to a silting object.  Our main technical step is
to transfer such a complex back to the module category. We regard a bounded complex of $A$-modules as a module over
$B=A\otimes_k R$, where $R$ is the radical-square-zero path algebra of a
linearly oriented quiver, and adjoin the appropriate projective disk
complexes. In this way we obtain
\[
  \begin{array}{c}
    \text{full-rank presilting complex }P^{\bullet},\\[-1pt]
    \thick(P^\bullet)\neq\Kb(\proj A)
  \end{array}
  \quad\longmapsto\quad
  \begin{array}{c}
    \text{faithful full-rank pretilting $B$-module }M,\\[-1pt]
    M\text{ has no tilting completion}.
  \end{array}
\]
The homological input is a natural comparison between ordinary extension
groups over the tensor-product algebra and morphisms in the homotopy
category.  This proves Theorem~\ref{thm:Cn}.  

Then a one-point extension construction produces the nonfaithful
self-orthogonal example, which proves
Theorem~\ref{thm:SFC-counterexample} and yields the universal equivalence
between the two self-orthogonality conjectures.
Another regular one-point extension
then converts the full-rank example into an almost-full-rank example.  A restriction argument shows that any hypothetical tilting completion over the
one-point extension would restrict to a tilting completion of $M$, which
gives Theorem~\ref{thm:Cnminusone}.

The paper is organized as follows. Section~2 realizes bounded complexes as
modules and establishes the Ext--homotopy comparison. Section~3 proves the
transfer theorem, including the rank count, faithfulness, and the
non-generation argument. Section~4 combines this construction with the work
of Krah, Hille--Perling, and Kalck to prove Theorem~\ref{thm:Cn}.
Section~5 derives consequences for Wakamatsu tilting modules,
Ext-perpendicular categories, and $\tau$-rigidity. Section~6 establishes the
one-point extension reduction, proves
Theorem~\ref{thm:SFC-counterexample}, and compares the two self-orthogonality
conjectures. Finally, Section~7 studies a regular one-point extension and
proves Theorem~\ref{thm:Cnminusone}.

\section{Complexes as modules and the Ext--homotopy comparison}

Let $k$ be a field, let $A$ be a finite-dimensional $k$-algebra, and fix an
integer $N\geq1$.  Consider the quiver
\[
  \Gamma=\Gamma_N:\quad
  0\xrightarrow{\alpha_0}1\xrightarrow{\alpha_1}\cdots
  \xrightarrow{\alpha_{N-1}}N
\]
and the radical-square-zero algebra
\[
  R=R_N
  =
  k\Gamma\big/
  \bigl\langle \alpha_{i+1}\alpha_i\mid 0\leq i\leq N-2\bigr\rangle.
\]
We compose paths from right to left.  Let $B=A\otimes_k R$.

Denote by $C^{[0,N]}(A\text{-}\mathrm{mod})$ the category of complexes
of finite-dimensional left $A$-modules supported in $[0,N]$, equipped with
its natural abelian structure.
Although the following identification is standard, we record it in our
conventions for later reference.

\begin{lemma}\label{lem:complex-module-equivalence}
There is a natural equivalence of abelian categories
\[
  B\text{-}\mathrm{mod}
  \simeq
  C^{[0,N]}(A\text{-}\mathrm{mod}).
\]
Under this equivalence, a left $B$-module $X$ corresponds to the complex
\[
X^\bullet=X^0\xrightarrow{d_X^0}X^1\longrightarrow\cdots
\xrightarrow{d_X^{N-1}}X^N,
\]
where $X^i=(1_A\otimes\varepsilon_i)X$ and
$d_X^i=(1_A\otimes\alpha_i)|_{X^i}$.
Moreover, $B$-module homomorphisms correspond to chain maps, and exact
sequences correspond to degreewise exact sequences.
\end{lemma}

\begin{proof}
This is the usual equivalence between modules over a bound quiver algebra
and representations of its bound quiver, applied in the category
$A\text{-}\mathrm{mod}$. Explicitly, the orthogonal idempotents
$\varepsilon_0,\ldots,\varepsilon_N$ give $X=\bigoplus_{i=0}^N(1_A\otimes\varepsilon_i)X$.
Since
$
\varepsilon_{i+1}\alpha_i=\alpha_i=\alpha_i\varepsilon_i,$
left multiplication by $1_A\otimes\alpha_i$ induces an $A$-linear map
$X^i\to X^{i+1}$, and the relations
$\alpha_{i+1}\alpha_i=0$ are precisely the identities
$d_X^{i+1}d_X^i=0$. The converse construction is immediate, and morphisms,
kernels, and cokernels are computed degreewise.
\end{proof}

By Lemma~\ref{lem:complex-module-equivalence}, we henceforth identify
$B$-modules with complexes of $A$-modules supported in $[0,N]$, and
extend all such complexes by zero outside this interval. We use the cohomological convention
$(X[t])^i=X^{i+t}$ and $d_{X[t]}^i=(-1)^t d_X^{i+t}$.

For a complex $X$ supported in $[0,N]$, define $\rho X$ by
$(\rho X)^0=0$ and $(\rho X)^i=X^{i-1}$ for $1\leq i\leq N$, with
differential $d_{\rho X}^i=-d_X^{i-1}$ for $1\leq i<N$.
Thus $\rho X$ is obtained from $X[-1]$ by deleting the term in degree
$N+1$.

For an $A$-module $U$ and $0\leq i<N$, let $D_i(U)$ be the disk complex
with $U$ in degrees $i$ and $i+1$, identity differential, and zero terms in
all other degrees.  Let $S_N(U)$ be the stalk complex with $U$ in degree
$N$.  If $U$ is projective over $A$, then $D_i(U)$ and $S_N(U)$ are
projective $B$-modules.  Indeed, $D_i(A)\cong B(1_A\otimes\varepsilon_i)$
for $0\leq i<N$ and $S_N(A)\cong B(1_A\otimes\varepsilon_N)$, and the
claim follows by taking finite direct sums and direct summands.

The following canonical presentation is the source of the comparison
result.

\begin{lemma}
\label{lem:canonical-presentation}
Let $X$ be a $B$-module.
Suppose that every term of $X$ is projective as an $A$-module.  Set
\[
  \mathcal Q(X)
  =
  \bigoplus_{i=0}^{N-1}D_i(X^i)\oplus S_N(X^N).
\]
Then $\mathcal Q(X)$ is a projective $B$-module, and there is a natural
short exact sequence
\begin{equation}
\label{eq:canonical-presentation}
  0\longrightarrow\rho X\longrightarrow\mathcal Q(X)
  \xrightarrow{\pi_X}X\longrightarrow0.
\end{equation}
\end{lemma}

\begin{proof}
The term of $\mathcal Q(X)$ in degree zero is $X^0$, and its term in degree
$i\geq1$ is $X^{i-1}\oplus X^i$. Define $\pi_X\colon \mathcal Q(X)\longrightarrow X$ by
$\pi_X^0=1_{X^0}$ and
$\pi_X^i(u,v)=d_X^{i-1}(u)+v$ for $1\leq i\leq N$.
A direct calculation shows that $\pi_X$ is a morphism of complexes, and it
is degreewise surjective.  For $1\leq i\leq N$, the map
\[
  X^{i-1}\longrightarrow\ker(\pi_X^i),
  \qquad
  u\longmapsto\bigl(u,-d_X^{i-1}(u)\bigr),
\]
is an isomorphism.  Thus the inclusion $\iota_X\colon\rho X\to\mathcal Q(X)$
is given by $\iota_X^0=0$ and
$\iota_X^i(u)=(u,-d_X^{i-1}(u))$ for $1\leq i\leq N$.  Under these
identifications, the differential induced on the kernel is $-d_X^{i-1}$.
Hence the kernel is naturally identified with $\rho X$, proving
\eqref{eq:canonical-presentation}.
\end{proof}

\begin{corollary}
\label{cor:projective-dimension}
If every term of a $B$-module $X$ is projective over $A$, then $\pd_B X\leq N$.
\end{corollary}

\begin{proof}
Apply Lemma~\ref{lem:canonical-presentation} successively to
$X,\rho X,\ldots,\rho^{N-1}X$.  The complex
$\rho^N X=S_N(X^0)$ is projective over $B$, so the resulting projective
resolution has length at most $N$.
\end{proof}

We now identify all positive extension groups.  The degree-one case is the
standard description of extensions in the exact category of complexes with
degreewise split exact sequences; see, for example,
\cite[Chapter~I]{Happel1988}.  We include the argument because the extension
groups below are computed in the ordinary module category of $B$.

\begin{proposition}
\label{prop:ext-comparison}
Let $X$ and $Y$ be $B$-modules, and suppose that every term of $X$ is
projective over $A$.  For every integer $r\geq1$, there is a natural
isomorphism
\begin{equation}
\label{eq:ext-comparison}
  \Ext_B^r(X,Y)
  \cong
  \Hom_{\Kb(A\text{-}\mathrm{mod})}(X,Y[r]).
\end{equation}
\end{proposition}

\begin{proof}
Consider an extension $0\longrightarrow Y\longrightarrow E\longrightarrow X\longrightarrow0$ of $B$-modules. Under Lemma~\ref{lem:complex-module-equivalence}, this becomes a degreewise exact sequence of complexes of $A$-modules. Since every $X^i$ is projective over $A$, the sequence $0\longrightarrow Y^i\longrightarrow E^i\longrightarrow X^i \longrightarrow0$ splits in $A\text{-}\mathrm{mod}$ for every $i$. Choose, for each $i$, an $A$-linear
splitting $s_i\colon X^i\to E^i$ of $\pi^i$, i.e., $\pi^i s_i=1_{X^i}$.
Then $E^i\cong Y^i\oplus X^i$, and in this decomposition the differential
of $E$ takes the form
$$
d_E^i=\begin{pmatrix} d_Y^i & f^i\\ 0 & d_X^i \end{pmatrix},
$$
where the off-diagonal term $f^i\colon X^i\to Y^{i+1}$ is uniquely
determined by $d_E^i s_i=s_{i+1}d_X^i+\iota_{i+1}f^i$. The complex
condition $d_E^{i+1}d_E^i=0$ amounts to $f^{i+1}d_X^i+d_Y^{i+1}f^i=0$, i.e., $f=(f^i)_i$ is a chain map $X\to Y[1]$.

Replacing $s_i$ by $s'_i=s_i+\iota_i h_i$ for an arbitrary $A$-linear map $h_i\colon X^i\to Y^i$ changes $f$ by $(f')^i=f^i+d_Y^i h_i-h_{i+1}d_X^i$, so $f'-f$ is null-homotopic; hence the homotopy class of $f$ in $\Hom_{\Kb(A\text{-}\mathrm{mod})}(X,Y[1])$ is independent of the choice of splittings. Conversely, every chain map $f\colon X\to Y[1]$ yields a termwise split extension $0\to Y\to E\to X\to 0$ by endowing $E^i=Y^i\oplus X^i$ with the differentials displayed above, the chain map condition being precisely the complex condition for $E$. 
Under the present hypothesis, every extension of $X$ by $Y$ is degreewise split as a sequence of $A$-modules. The two constructions above are inverse and therefore $\Ext_B^1(X,Y)\cong\Hom_{\Kb(A\text{-}\mathrm{mod})}(X,Y[1])$.

This correspondence respects Baer sums and is natural in both variables.

For $1\leq r\leq N$, dimension shifting along
\eqref{eq:canonical-presentation} gives
\[
  \Ext_B^r(X,Y)
  \cong
  \Ext_B^1(\rho^{r-1}X,Y)
  \cong
  \Hom_{\Kb(A\text{-}\mathrm{mod})}
  \bigl(\rho^{r-1}X,Y[1]\bigr).
\]
The complex $\rho^{r-1}X$ is the brutal truncation in degrees at most $N$
of $X[-r+1]$.  Since $Y$ is supported in $[0,N]$, the shifted complex
$Y[1]$ is supported in $[-1,N-1]$ and in particular is zero in every degree
at least $N$.  Consequently, restriction and extension by zero induce a bijection between chain maps $X[-r+1]\longrightarrow Y[1]$ and $\rho^{r-1}X\longrightarrow Y[1]$. The same assertion holds for homotopies: a homotopy component in degree
$i$ has target $Y^i$, and therefore vanishes for $i>N$.  It follows that
\[
  \Hom_{\Kb(A\text{-}\mathrm{mod})}
  \bigl(\rho^{r-1}X,Y[1]\bigr)
  \cong
  \Hom_{\Kb(A\text{-}\mathrm{mod})}
  \bigl(X[-r+1],Y[1]\bigr)
  \cong
  \Hom_{\Kb(A\text{-}\mathrm{mod})}(X,Y[r]).
\]

If $r>N$, then the extension group vanishes by
Corollary~\ref{cor:projective-dimension}.  The homotopy group also vanishes,
since $X$ is supported in $[0,N]$ while $Y[r]$ is supported in
$[-r,N-r]$, and these intervals are disjoint.  This proves
\eqref{eq:ext-comparison} for every $r\geq1$.
\end{proof}

\section{From presilting complexes to pretilting modules}

We now turn the preceding comparison into a transfer theorem.  

\begin{theorem}
\label{thm:transfer}
Let $k$ be algebraically closed, let $A$ be a finite-dimensional basic
$k$-algebra, and let $P^\bullet\in\Kb(\proj A)$ be a basic presilting complex such
that
\[
  |P^\bullet|=n(A)
  \qquad\text{and}\qquad
  \thick(P^\bullet)\neq\Kb(\proj A).
\]
Then there exist a finite-dimensional basic $k$-algebra $B$ and a faithful basic pretilting $B$-module $M$ such that $|M|=n(B)$, and $M$ is not a direct summand of any tilting $B$-module.
\end{theorem}

\begin{proof}
Set $n=n(A)$. Denote by $P^\bullet_1,\ldots,P^\bullet_n$ the direct summands of $P^\bullet$. We may assume that they are pairwise non-isomorphic minimal projective complexes.  After a common shift,
assume that $P^\bullet$ is supported in the smallest interval $[0,N]$.  We have
$N\geq1$.  Indeed, if $N=0$, then $P^\bullet$ is a basic projective module with
$n$ indecomposable summands.  It is therefore isomorphic to the basic
regular module $A$ and is silting, contrary to the hypothesis.

Let $R=R_N$ be the algebra from Section~2, put $B=A\otimes_k R$, and regard
$P^\bullet=\oplus_{j=1}^nP_j^{\bullet}$ as a $B$-module.  Denote
this module by $P=\oplus_{j=1}^nP_j$.  Choose a complete set of pairwise orthogonal
primitive idempotents $e_1,\ldots,e_n$ of $A$, so that
$1_A=e_1+\cdots+e_n$, and set
\[
  Q=
  \bigoplus_{i=0}^{N-1}\bigoplus_{j=1}^n D_i(Ae_j),
  \qquad
  M=P\oplus Q.
\]

Let $\varepsilon_0,\ldots,\varepsilon_N$ be the vertex idempotents of
$R$.  The ideal
$(\rad A)\otimes_k R+A\otimes_k(\rad R)$
is nilpotent, and the corresponding quotient of $B$ is
$(A/\rad A)\otimes_k(R/\rad R)\cong k^{(N+1)n}$.
Thus this ideal is $\rad B$, the algebra $B$ is basic, and
$n(B)=(N+1)n$. Its indecomposable projective left modules are
$D_i(Ae_j)$ for $0\leq i<N$ and $1\leq j\leq n$, and
$S_N(Ae_j)$ for $1\leq j\leq n$.

We next check that $M$ is basic. If some $P_j$
decomposed as a direct sum of two nonzero $B$-modules, then it would
decompose as a direct sum of two nonzero minimal complexes.  Neither
summand is zero in $\Kb(\proj A)$, contradicting the indecomposability of
$P_j$.  Hence the $P_j$ remain indecomposable as $B$-modules.  A
$B$-module isomorphism between two of them is a chain isomorphism, so they
remain pairwise non-isomorphic. Finally, no $P_j$ is isomorphic to a
summand of $Q$, because every summand of $Q$ is contractible as a complex,
whereas $P_j$ is a nonzero minimal complex. Therefore $M$ is basic and
\[
  |M|=n+Nn=(N+1)n=n(B).
\]

Corollary~\ref{cor:projective-dimension} gives
$\pd_BP\leq N$, while $Q$ is projective over $B$.  For every
$r>0$, Proposition~\ref{prop:ext-comparison} gives
\[
  \Ext_B^r(P,P)
  \cong
  \Hom_{\Kb(\proj A)}(P,P[r])
  =0
\]
because $P$ is presilting. The underlying complex of $Q$ is contractible, so the same proposition gives $\Ext_B^r(P,Q)=0$. Since $Q$ is projective as a $B$-module, we also have $\Ext_B^r(Q,M)=0$. Thus $M$ is pretilting.

It remains to prove faithfulness.  
Put $\varepsilon=\varepsilon_0+\cdots+\varepsilon_{N-1}$, noting that $\varepsilon_N$ does not appear here.
Then
\[
  Q\cong B(1_A\otimes\varepsilon)
  \cong A\otimes_k R\varepsilon.
\]
To show that $Q$ is faithful it suffices to prove that the action map
$$
  \lambda:\ B=A\otimes_k R\longrightarrow
  \operatorname{End}_k(Q)=\operatorname{End}_k(A\otimes_k R\varepsilon),
  \qquad
  b\longmapsto(q\mapsto bq),
$$
is injective.  Note that
this action is factorwise: for $a\in A$, $r\in R$, $u\in A$ and
$v\in R\varepsilon$,
$$
  (a\otimes r)\cdot(u\otimes v)=au\otimes rv.
$$
The regular left $A$-module is faithful. We show that the left $R$-module
$R\varepsilon$ is faithful. Indeed, suppose that
$xR\varepsilon=0$.  Since $\varepsilon_i\in R\varepsilon$ for $i<N$,
the path-basis expansion of $x$ shows that $x=c\varepsilon_N$ for some
$c\in k$.  But $\alpha_{N-1}\in R\varepsilon$ and
$\varepsilon_N\alpha_{N-1}=\alpha_{N-1}$, so $c=0$. Hence the left regular
representation $A\to\operatorname{End}_k(A)$
and the representation $R\to\operatorname{End}_k(R\varepsilon)$ are both injective.  The action map $\lambda$
factors as
$$
  A\otimes_k R
  \hookrightarrow
  \operatorname{End}_k(A)\otimes_k
  \operatorname{End}_k(R\varepsilon)
  \xrightarrow{\sim}
  \operatorname{End}_k(A\otimes_k R\varepsilon),
$$
where the first arrow is the tensor product of the two injective
representations above, and the second arrow is the canonical map
$\phi\otimes\psi\mapsto\bigl(u\otimes v\mapsto\phi(u)\otimes\psi(v)\bigr)$,
which is bijective because $A$ and $R\varepsilon$ are finite-dimensional
over $k$. Thus $\lambda$ is injective and $Q$ is a faithful $B$-module.

Suppose that $M$ were tilting.  There would be an exact sequence
\[
  0\longrightarrow B\longrightarrow M_0\longrightarrow M_1
  \longrightarrow\cdots\longrightarrow M_s\longrightarrow0
\]
with $M_i\in\add(M)$ for every $i$. Let \(\mathcal U:B\text{-}\mathrm{mod}\longrightarrow\Cb(A\text{-}\mathrm{mod})\) be the exact functor which sends a $B$-module to its underlying complex of $A$-modules. Decomposing the preceding exact sequence into short exact sequences and passing to the associated triangles gives \(\mathcal U(B)\in\thick_{\Db(A\text{-}\mathrm{mod})}(\mathcal U(M))\). Since $\mathcal U(Q)$ is contractible, it is zero in the derived category, and hence $\mathcal U(M)\cong P^\bullet$ in $\Db(A\text{-}\mathrm{mod})$.  Moreover, both $\mathcal U(B)$ and $P^\bullet$ are perfect. The canonical functor $\Kb(\proj A)\to\Db(A\text{-}\mathrm{mod})$ is fully faithful and identifies $\Kb(\proj A)$ with the perfect subcategory.  Therefore the preceding inclusion gives \(\mathcal U(B)\in\thick_{\Kb(\proj A)}(P^\bullet)\).
On the other hand, the decomposition of the regular $R$-module into its
indecomposable projectives gives
\[
  \mathcal U(B)
  \cong
  \bigoplus_{i=0}^{N-1}D_i(A)\oplus S_N(A)
  \cong
  A[-N]
\]
in $\Kb(\proj A)$.  Hence $A\in\thick(P^\bullet)$, and therefore
$\thick(P^\bullet)=\Kb(\proj A)$, a contradiction.  Thus $M$ is not tilting.


Finally, suppose that $M$ were a direct summand of a tilting $B$-module
$T$. Let $T^{\mathrm b}$ be the basic representative of $T$; then
$T^{\mathrm b}$ is again tilting and $\add(T^{\mathrm b})=\add(T)$. Since $M$ is
basic, every indecomposable direct summand of $M$ occurs in
$T^{\mathrm b}$. Both $M$ and $T^{\mathrm b}$ have exactly $n(B)$ pairwise
non-isomorphic indecomposable direct summands
\cite[Theorem~1.19]{Miyashita1986}. Hence
\[
  \add(T)=\add(T^{\mathrm b})=\add(M).
\]
The tilting coresolution for $T$ is therefore a coresolution by objects in
$\add(M)$. Since $M$ already has finite projective dimension and is
self-orthogonal, it follows that $M$ is tilting, contradicting the preceding
paragraph. Thus $M$ is not a direct summand of any tilting $B$-module.
\end{proof}

For the geometric application, we also need the following permanence
property.

\begin{lemma}
\label{lem:quasi-hereditary}
If $A$ is quasi-hereditary, then the algebra
$B=A\otimes_k R$ constructed in Theorem~\ref{thm:transfer} is
quasi-hereditary.
\end{lemma}

\begin{proof}
The algebra $R$ is directed and hence quasi-hereditary
\cite{DlabRingel1992}.  The tensor product of two quasi-hereditary algebras
is quasi-hereditary with respect to the product order on their weight
posets \cite[Section~2]{Chan2014}.  Therefore $A\otimes_k R$ is
quasi-hereditary.
\end{proof}

\section{The full-rank counterexample}

We recall the full-rank presilting complex furnished by the work of Krah,
Hille--Perling, and Kalck. Let $X$ be the blow-up of
$\mathbb P^2_{\C}$ at ten points in general position. Krah
\cite{Krah2024} constructed a non-full exceptional sequence
$(E_1,\ldots,E_{13})$ of line bundles on $X$, and
$K_0(X)\cong\mathbb Z^{13}$. As explained by Kalck
\cite{Kalck2023}, after suitably shifting the $E_i$, their direct sum is a
basic presilting object $E$ with thirteen pairwise non-isomorphic
indecomposable direct summands. Since shifting does not change the thick
subcategory generated by the $E_i$ and the exceptional sequence is not full,
\[
  \thick(E)\neq\Db(\operatorname{coh}X).
\]

By Hille and Perling \cite{HillePerling2014}, the surface $X$ admits a
tilting bundle whose endomorphism algebra is quasi-hereditary. Choose such a
tilting bundle $\mathcal T_0$, let $\mathcal T$ be its basic representative,
and set
\[
  A=\operatorname{End}_{X}(\mathcal T)^{\mathrm{op}}.
\]
Since $\add(\mathcal T)=\add(\mathcal T_0)$, the bundle $\mathcal T$ is
still tilting, and its endomorphism algebra is Morita equivalent to that of
$\mathcal T_0$. The opposite is needed because
$\mathbf R\!\operatorname{Hom}_{X}(\mathcal T,-)$ naturally takes values
in right $\operatorname{End}_{X}(\mathcal T)$-modules, equivalently in left
$A$-modules. Since quasi-heredity is Morita invariant and is preserved under
passage to the opposite algebra, $A$ is a finite-dimensional basic
quasi-hereditary $\C$-algebra; in particular, it has finite global dimension.
Moreover, $A$ is connected: a nontrivial product decomposition of $A$ would
induce a nontrivial orthogonal direct-sum decomposition of
$\Db(\operatorname{coh}X)$ and hence a decomposition of $X$ into nonempty
open-and-closed subschemes, contrary to the connectedness of $X$. Thus the
tilting functor yields triangle equivalences
\[
  \Db(\operatorname{coh}X)
  \xrightarrow{\sim}
  \Db(A\text{-}\mathrm{mod})
  \xrightarrow{\sim}
  \Kb(\proj A).
\]
It also induces an isomorphism on Grothendieck groups. Since the rank of
$K_0(A\text{-}\mathrm{mod})$ equals the number $n(A)$ of isomorphism classes
of simple $A$-modules, we obtain
\[
  n(A)=\operatorname{rank}K_0(A\text{-}\mathrm{mod})
      =\operatorname{rank}K_0(X)=13.
\]
Transporting $E$ along the displayed equivalences gives a basic presilting
complex $P^\bullet\in\Kb(\proj A)$ such that
\[
  |P^\bullet|=13=n(A)
  \qquad\text{and}\qquad
  \thick(P^\bullet)\neq\Kb(\proj A).
\]
Summarizing the preceding discussion, we have the following result; see
\cite{Kalck2023}.

\begin{proposition}
\label{prop:geometric-input}
There exist a finite-dimensional basic connected quasi-hereditary $\C$-algebra $A$
and a basic presilting complex $P^\bullet\in\Kb(\proj A)$ such that
\[
  |P^\bullet|=n(A)=13
  \qquad\text{and}\qquad
  \thick(P^\bullet)\neq\Kb(\proj A).
\]
\end{proposition}

\begin{proof}[Proof of Theorem~\ref{thm:Cn}]
Apply Theorem~\ref{thm:transfer} to the algebra and presilting complex in
Proposition~\ref{prop:geometric-input}.  This produces a finite-dimensional
basic $\C$-algebra $B$ and a faithful basic pretilting $B$-module $M$ with
$|M|=n(B)$ which is not a direct summand of any tilting $B$-module.
Lemma~\ref{lem:quasi-hereditary} shows that $B$ is quasi-hereditary.
Finally, $B$ is connected.  Indeed, the description of $\rad B$ in the proof
of Theorem~\ref{thm:transfer} shows that the Gabriel quiver of
$B=A\otimes_k R$ contains, at each vertex of the connected quiver of $R$, a
copy of the Gabriel quiver of $A$, and, at each vertex of the Gabriel quiver
of $A$, a copy of the connected quiver of $R$.  Since $A$ and $R$ are
connected, so is $B$.
\end{proof}

\section{Consequences for Wakamatsu tilting modules}
\label{sec:wakamatsu-consequences}
We next place the full-rank example in Enomoto's framework.  Note that Enomoto uses
finitely generated right modules; throughout this section we use the
corresponding left-module terminology, obtained by passing to opposite
algebras.  For a self-orthogonal left $A$-module $T$, put
\[
  {}^{\perp}T
  =
  \{X\in A\text{-}\mathrm{mod}\mid
    \Ext_A^i(X,T)=0\text{ for every }i>0\}.
\]
Let $\mathcal X_T$ be the class of modules $X$ admitting an exact sequence
\[
  0\longrightarrow X\longrightarrow T^0\longrightarrow T^1
  \longrightarrow T^2\longrightarrow\cdots
\]
with $T^i\in\add(T)$ and with every successive image in ${}^{\perp}T$.
Following \cite[Definition~2.7]{Enomoto2023}, the module $T$ is Wakamatsu
tilting if it is self-orthogonal and $A\in\mathcal X_T$.  We also recall that
a subcategory $\mathcal C\subseteq A\text{-}\mathrm{mod}$ has a
\emph{finite cover} if there exists $C\in\mathcal C$ such that every object
of $\mathcal C$ is an epimorphic image of an object in $\add(C)$; see
\cite[Definition~2.11]{Enomoto2023}.

\begin{lemma}
\label{lem:wakamatsu-finite-global-dimension}
Let $A$ be a finite-dimensional algebra of finite global dimension.  Then every
Wakamatsu tilting module is a tilting module.
\end{lemma}
\begin{proof}
Let $T$ be Wakamatsu tilting. Since
$\operatorname{gl.dim}A<\infty$, one has
$\pd_A T<\infty$ and
$\operatorname{findim}A<\infty$, where $\operatorname{findim}A=\mathrm{sup}\{\pd_A M | \pd_A M<\infty\}$ is the finitistic dimension of $A$. Hence the result follows from
\cite[Proposition~4.4(i)]{ManteseReiten2004}.
\end{proof}

\begin{corollary}
\label{cor:enomoto-consequences}
Let $B$ and $M$ be as in Theorem~\ref{thm:Cn}, and set
\[
  M^{\perp}
  =
  \{X\in B\text{-}\mathrm{mod}\mid
    \Ext_B^i(M,X)=0\text{ for every }i>0\}.
\]
Then the following statements hold.
\begin{enumerate}[label=\textnormal{(\arabic*)},leftmargin=2.5em]
  \item The module $M$ is not Wakamatsu tilting.
  \item The module $M$ has no Bongartz completion in the sense of
  \cite[Definition~3.9]{Enomoto2023}.
  \item The category $M^{\perp}$ has no finite cover and is not covariantly
  finite in $B\text{-}\mathrm{mod}$.
\end{enumerate}
In particular, $M$ gives a negative answer to \cite[Conjecture~5.9]{Enomoto2023}.
\end{corollary}

\begin{proof}
Since $B$ is quasi-hereditary, it has finite global dimension.  If $M$
were Wakamatsu tilting, then it would be tilting by
Lemma~\ref{lem:wakamatsu-finite-global-dimension}, contrary to
Theorem~\ref{thm:Cn}.  This proves \textnormal{(1)}.

Suppose that $M$ admits a Bongartz completion $T$.  Since
$\pd_B M<\infty$, \cite[Proposition~3.10]{Enomoto2023} implies that
$T$ is a tilting module. So $M$ is partial-tilting, which contradicts Theorem~\ref{thm:Cn}.  Thus \textnormal{(2)}
holds.

By \cite[Theorem~3.12]{Enomoto2023}, a self-orthogonal module admits a
Bongartz completion if and only if its Ext-perpendicular category has a
finite cover.  Hence \textnormal{(2)} implies that $M^{\perp}$ has no finite
cover.  Since $M^{\perp}$ is closed under extensions and direct summands,
\cite[Proposition~2.12(2)]{Enomoto2023} shows that covariant finiteness
would imply the existence of a finite cover.  Therefore $M^{\perp}$ is not
covariantly finite, proving \textnormal{(3)}.

The final assertion follows from Theorem~\ref{thm:Cn} and
\textnormal{(1)}.
\end{proof}

\begin{remark}
\label{rem:not-wakamatsu-tilting-conjecture}
Corollary~\ref{cor:enomoto-consequences} is not a counterexample to the
Wakamatsu tilting conjecture, which starts with a Wakamatsu tilting module of
finite projective dimension and predicts that it is tilting.  Here finite
global dimension makes that implication valid by
Lemma~\ref{lem:wakamatsu-finite-global-dimension}, and this is precisely what
forces $M$ not to be Wakamatsu tilting.  Likewise, statement~\textnormal{(2)}
concerns Enomoto's projectively Wakamatsu tilting Bongartz completion,
characterized by equality of Ext-perpendicular categories.  It should not be
rephrased as the nonexistence of an arbitrary Wakamatsu tilting module merely
containing $M$ as a direct summand.
\end{remark}

\begin{remark}
\label{rem:tau-rigidity}
The module $M$ in Theorem~\ref{thm:Cn} is not $\tau$-rigid.  Indeed,
suppose that $M$ were $\tau$-rigid.  Since $M$ is basic and
$|M|=n(B)$, it would then be a $\tau$-tilting $B$-module \cite{AdachiIyamaReiten2014}. Then Zhang's theorem
\cite[Theorem~1.3]{Zhang2022}, together with $\pd_B M<\infty$ and $\Ext_B^i(M,M)=0$ for every $i>0$, would imply that $M$ is a classical tilting module, contradicting Theorem~\ref{thm:Cn}.  Thus $M$
is neither $\tau$-rigid nor $\tau$-tilting.
\end{remark}

\section{From failure of Wakamatsu tilting to nonfaithfulness}
\label{sec:nonfaithful-full-rank}
Chen, Li, Zhang, and Zhao \cite[Section~5]{ChenLiZhangZhao2025}
formulated the \emph{Self-orthogonal Faithful Conjecture}: every
self-orthogonal module $T$ satisfying $|T|=n(A)$ is faithful. They observed
that, for each fixed algebra, Enomoto's Self-orthogonal Wakamatsu-tilting
Conjecture implies the Self-orthogonal Faithful Conjecture, since every
Wakamatsu tilting module is faithful.

We now use the results of the preceding section to construct counterexamples
to the Self-orthogonal Faithful Conjecture.  The first step is to extract,
from the failure of the Wakamatsu-tilting property, a suitable
nonmonomorphic approximation.  This will provide the data for the
one-point extension construction below.

\subsection{A nonmonomorphic approximation}

For a $B$-module $M$, recall the notation
\[
{}^{\perp}M=
\{X\in B\text{-}\mathrm{mod}\mid
\Ext_B^i(X,M)=0\text{ for every }i>0\}.
\]

\begin{lemma}
\label{lem:nonmonic-left-approximation}
Let $B$ be a finite-dimensional $k$-algebra and let $M$ be a self-orthogonal
$B$-module.  If $M$ is not Wakamatsu tilting, then there exist a module
$U\in{}^{\perp}M$ and a minimal left $\add(M)$-approximation
$g\colon U\longrightarrow M^0$, with $M^0\in\add(M)$, which is not a
monomorphism.
\end{lemma}

\begin{proof}
Set $X_0=B$.  Since $B$ is projective, $X_0\in{}^{\perp}M$.  Suppose that
$X_i\in{}^{\perp}M$ has been constructed.  Since $\add(M)$ is covariantly
finite in $B\text{-}\mathrm{mod}$, there is a left $\add(M)$-approximation
$X_i\to M^i$; after deleting redundant direct summands of $M^i$, we may take
it to be minimal.  Denote it by $g_i\colon X_i\longrightarrow M^i$. If $g_i$ is a monomorphism, set
$X_{i+1}=\operatorname{Coker}g_i$, so that
\begin{equation}
\label{eq:approximation-step}
  0\longrightarrow X_i\xrightarrow{g_i}M^i
  \longrightarrow X_{i+1}\longrightarrow0
\end{equation}
is exact.

We claim that $X_{i+1}\in{}^{\perp}M$. Applying $\Hom_B(-,M)$ to
\eqref{eq:approximation-step}, the approximation property makes the map \(\Hom_B(M^i,M)\longrightarrow\Hom_B(X_i,M)\) surjective.  Hence $\Ext_B^1(X_{i+1},M)=0$.  For $r\geq2$, the same long
exact sequence, together with
\[
  \Ext_B^{>0}(M^i,M)=0
  \qquad\text{and}\qquad
  \Ext_B^{>0}(X_i,M)=0,
\]
gives $\Ext_B^r(X_{i+1},M)=0$.  Thus
$X_{i+1}\in{}^{\perp}M$.

If every $g_i$ were a monomorphism, then splicing the sequences
\eqref{eq:approximation-step} would give an exact sequence
\[
  0\longrightarrow B\longrightarrow M^0\longrightarrow M^1
  \longrightarrow M^2\longrightarrow\cdots
\]
whose terms belong to $\add(M)$.  The image of $M^{i-1}\to M^i$ is
isomorphic to $X_i$, and hence belongs to ${}^{\perp}M$.  Together with the
self-orthogonality of $M$, this is precisely the assertion that $M$ is
Wakamatsu tilting.  Hence some $g_i$ is not a monomorphism.  Taking
$U=X_i$, $M^0=M^i$, and $g=g_i$ proves the lemma.
\end{proof}

\subsection{Extension calculus for one-point extensions}

The nonmonomorphic approximation furnished by
Lemma~\ref{lem:nonmonic-left-approximation} will be incorporated into a
one-point extension of $B$.  To carry out this construction and verify
self-orthogonality of the resulting module, we first record the required
extension calculus for one-point extensions.
Let $U$ be a left $B$-module and set
\begin{equation}
\label{eq:one-point extension-general}
  \Gamma=
  \begin{pmatrix}
    B&U\\
    0&k
  \end{pmatrix},
  \qquad
  e=
  \begin{pmatrix}
    1_B&0\\
    0&0
  \end{pmatrix},
  \qquad
  f=
  \begin{pmatrix}
    0&0\\
    0&1
  \end{pmatrix}.
\end{equation}
Here $U$ is regarded as a $B$--$k$-bimodule in the evident way.  A left
$\Gamma$-module is identified with a triple
\[
  (X,V,\varphi),
  \qquad
  X\in B\text{-}\mathrm{mod},\quad
  V\in k\text{-}\mathrm{mod},\quad
  \varphi\colon U\otimes_kV\longrightarrow X,
\]
with action
\[
  \begin{pmatrix}b&u\\0&c\end{pmatrix}(x,v)
  =
  \bigl(bx+\varphi(u\otimes v),cv\bigr).
\]
A morphism $(a,s)\colon (X,V,\varphi)\to(X',V',\varphi')$ consists of maps
$a\colon X\to X'$ and $s\colon V\to V'$ satisfying
\[
  a\varphi=\varphi'(1_U\otimes s).
\]
We write
\[
  \iota(X)=(X,0,0),
  \qquad
  S=(0,k,0).
\]
Then $\Gamma f=(U,k,1_U)$ is projective, and the inclusion in the first
component and the projection onto the second component give a canonical exact
sequence
\begin{equation}
\label{eq:new-simple-presentation}
  0\longrightarrow\iota(U)\longrightarrow\Gamma f
  \longrightarrow S\longrightarrow0.
\end{equation}

The following extension formulas are standard for one-point extensions.
We include a short proof to identify explicitly the connecting map used in the proof
of Theorem~\ref{thm:SWC-to-SFC-reduction}.

\begin{lemma}
\label{lem:one-point extension-ext-calculus}
With the notation above, for all $B$-modules $X,Y$ and all $r\geq0$ one has
\begin{align}
  \Ext_\Gamma^r(\iota X,\iota Y)
  &\cong \Ext_B^r(X,Y),
  \label{eq:iota-ext}\\
  \Ext_\Gamma^r(\iota X,S)
  &=0.
  \label{eq:iota-to-simple-ext}
\end{align}
Moreover,
\begin{align}
  \Ext_\Gamma^1(S,\iota Y)
  &\cong\Hom_B(U,Y),
  \label{eq:simple-ext-one}\\
  \Ext_\Gamma^r(S,\iota Y)
  &\cong\Ext_B^{r-1}(U,Y)
  \qquad (r\geq2),
  \label{eq:simple-ext-higher}\\
  \Ext_\Gamma^r(S,S)&=0
  \qquad (r>0).
  \label{eq:simple-self-ext}
\end{align}
The isomorphism in \eqref{eq:simple-ext-one} may be chosen so that a map
$h\colon U\to Y$ corresponds to the extension
\begin{equation}
\label{eq:extension-associated-to-h}
  0\longrightarrow\iota(Y)\longrightarrow(Y,k,h)
  \longrightarrow S\longrightarrow0.
\end{equation}
Consequently, if $g:U\to M^0$ is a $B$-linear map and
\begin{equation}
\label{eq:extension-associated-to-g}
  0\longrightarrow\iota(M^0)\longrightarrow C_g
  \longrightarrow S\longrightarrow0,
  \qquad
  C_g=(M^0,k,g),
\end{equation}
then the connecting map obtained by applying $\Hom_\Gamma(-,\iota Y)$ to
\eqref{eq:extension-associated-to-g} is, under
\eqref{eq:simple-ext-one}, the map
\begin{equation}
\label{eq:connecting-map-composition}
  \Hom_B(M^0,Y)\longrightarrow\Hom_B(U,Y),
  \qquad
  a\longmapsto ag.
\end{equation}
\end{lemma}

\begin{proof}
The functor $\iota$ is exact and sends projective $B$-modules to
projective $\Gamma$-modules; indeed, \(\iota(Be')\cong \Gamma e'\) for every idempotent $e'$ of $B$.  Hence a projective $B$-resolution of
$X$ is sent to a projective $\Gamma$-resolution of $\iota X$, which gives
\eqref{eq:iota-ext}.  Since \(\Hom_\Gamma(\iota P,S)=0\) for every projective $B$-module $P$, the same resolution gives
\eqref{eq:iota-to-simple-ext}.

Apply $\Hom_\Gamma(-,\iota Y)$ to
\eqref{eq:new-simple-presentation}.  Since $\Gamma f$ is projective and \(\Hom_\Gamma(\Gamma f,\iota Y)=0\), the connecting homomorphism yields a natural isomorphism
\[
  \delta_Y:
  \Hom_B(U,Y)
  \xrightarrow{\sim}
  \Ext_\Gamma^1(S,\iota Y),
\]
and dimension shifting gives
\eqref{eq:simple-ext-higher}.  Applying $\Hom_\Gamma(-,S)$ instead, and
using \eqref{eq:iota-to-simple-ext}, gives
\eqref{eq:simple-self-ext}.

Under $\delta_Y$, a map $h:U\to Y$ is represented by the pushout of
\eqref{eq:new-simple-presentation} along $\iota(h)$, which is precisely
\eqref{eq:extension-associated-to-h}.  In particular,
\eqref{eq:extension-associated-to-g} represents $\delta_{M^0}(g)$.
Naturality of the connecting homomorphism with respect to a map
$a\colon M^0\to Y$ therefore identifies the connecting map associated with
\eqref{eq:extension-associated-to-g} with
\eqref{eq:connecting-map-composition}.
\end{proof}

\subsection{The reduction theorem and its consequences}

We can now combine the two preceding lemmas.  Starting with a full-rank
self-orthogonal module which is not Wakamatsu tilting, we use
Lemma~\ref{lem:nonmonic-left-approximation} to obtain a nonmonomorphic
approximation $g$, and encode $g$ in a module over the corresponding
one-point extension.  The extension formulas above then allow us to
preserve self-orthogonality while the nontrivial kernel of $g$ forces the
resulting module to be nonfaithful.

\begin{theorem}
\label{thm:SWC-to-SFC-reduction}
Let $B$ be a finite-dimensional $k$-algebra and let $M$ be a basic
self-orthogonal $B$-module such that \(|M|=n(B)\).
If $M$ is not Wakamatsu tilting, then there exist a finite-dimensional
$k$-algebra $\Gamma$ and a basic self-orthogonal $\Gamma$-module $T$ such
that
\[
  |T|=n(\Gamma)
  \qquad\text{and}\qquad
  \operatorname{Ann}_\Gamma(T)\neq0.
\]
More precisely, one may take $\Gamma$ to be a one-point extension of $B$.
Each of the properties of being basic, connected, and quasi-hereditary is
inherited by $\Gamma$ whenever it is possessed by $B$.
\end{theorem}

\begin{proof}
By Lemma~\ref{lem:nonmonic-left-approximation}, choose a nonmonomorphic
minimal left $\add(M)$-approximation
\[
  g:U\longrightarrow M^0,
  \qquad
  U\in{}^{\perp}M.
\]
Let $\Gamma$ be the one-point extension
\eqref{eq:one-point extension-general}, put
\[
  C_g=(M^0,k,g),
  \qquad
  T=\iota(M)\oplus C_g,
\]
and retain the exact sequence \eqref{eq:extension-associated-to-g}.

We first prove self-orthogonality.  By \eqref{eq:iota-ext},
\[
  \Ext_\Gamma^r(\iota M,\iota M)=0
  \qquad (r>0).
\]
Since $M^0\in\add(M)$, applying $\Hom_\Gamma(\iota M,-)$ to
\eqref{eq:extension-associated-to-g} and using
\eqref{eq:iota-to-simple-ext} gives
\begin{equation}
\label{eq:iota-M-to-Cg}
  \Ext_\Gamma^r(\iota M,C_g)=0
  \qquad (r>0).
\end{equation}

Let $Y\in\add(M)$.  Apply $\Hom_\Gamma(-,\iota Y)$ to
\eqref{eq:extension-associated-to-g}.  In degree one, the relevant part of
the long exact sequence is
\[
  \Hom_B(M^0,Y)\longrightarrow\Hom_B(U,Y)
  \longrightarrow\Ext_\Gamma^1(C_g,\iota Y)
  \longrightarrow\Ext_B^1(M^0,Y).
\]
By Lemma~\ref{lem:one-point extension-ext-calculus}, the first map is
$a\mapsto ag$; it is surjective because $g$ is a left $\add(M)$-approximation.
The last term is zero by self-orthogonality.  Hence
$\Ext_\Gamma^1(C_g,\iota Y)=0$.  For $r\geq2$, the same long exact sequence,
\eqref{eq:simple-ext-higher}, and $U\in{}^{\perp}M$ give
\[
  \Ext_\Gamma^r(C_g,\iota Y)
  \cong
  \Ext_\Gamma^r(S,\iota Y)
  \cong
  \Ext_B^{r-1}(U,Y)
  =0.
\]
Thus
\begin{equation}
\label{eq:Cg-to-add-M}
  \Ext_\Gamma^r(C_g,\iota Y)=0
  \qquad (Y\in\add(M),\ r>0).
\end{equation}
In particular, this holds for $Y=M$ and $Y=M^0$.

Applying $\Hom_\Gamma(-,S)$ to
\eqref{eq:extension-associated-to-g}, and using
\eqref{eq:iota-to-simple-ext} and \eqref{eq:simple-self-ext}, yields
\begin{equation}
\label{eq:Cg-to-S}
  \Ext_\Gamma^r(C_g,S)=0
  \qquad (r>0).
\end{equation}
Finally, applying $\Hom_\Gamma(C_g,-)$ to
\eqref{eq:extension-associated-to-g} and using
\eqref{eq:Cg-to-add-M} for $Y=M^0$ together with
\eqref{eq:Cg-to-S}, we obtain
\[
  \Ext_\Gamma^r(C_g,C_g)=0
  \qquad (r>0).
\]
Together with \eqref{eq:iota-M-to-Cg} and
\eqref{eq:Cg-to-add-M} for $Y=M$, this proves that $T$ is self-orthogonal.

We next verify the rank and indecomposability assertions.  The radical of $\Gamma$ is
\[
  \rad\Gamma
  =
  \begin{pmatrix}
    \rad B&U\\
    0&0
  \end{pmatrix}.
\]
Indeed, the displayed ideal is nilpotent and its quotient is the semisimple
algebra
\[
  (B/\rad B)\times k.
\]
Consequently, $n(\Gamma)=n(B)+1$, and $\Gamma$ is basic whenever $B$ is
basic.  The functor $\iota$ preserves indecomposability and isomorphism
classes.

We claim that $C_g$ is indecomposable.  An endomorphism of $C_g$ is a pair
$(a,\lambda)$ with
\[
  a\in\operatorname{End}_B(M^0),
  \qquad
  \lambda\in k,
  \qquad
  ag=\lambda g.
\]
If $(a,\lambda)$ is idempotent, then $\lambda\in\{0,1\}$.  If $\lambda=1$,
then $ag=g$, so the left minimality of $g$ implies that $a$ is invertible;
since $a$ is idempotent, $a=1$.  If $\lambda=0$, then
$(1-a)g=g$, and the same argument gives $1-a=1$, hence $a=0$.  Thus
$C_g$ has no nontrivial idempotent endomorphism and is indecomposable.
Moreover, $fC_g\cong k$, whereas $f\iota(X)=0$ for every $B$-module $X$.
It follows that $C_g$ is not isomorphic to a summand of $\iota(M)$.  Hence
$T$ is basic and
\[
  |T|=|M|+1=n(B)+1=n(\Gamma).
\]

Since $g$ is not a monomorphism, choose $0\neq u\in\ker g$ and set
\[
  z_u=
  \begin{pmatrix}
    0&u\\
    0&0
  \end{pmatrix}
  \in\Gamma.
\]
The element $z_u$ annihilates $\iota(M)$.  Its action on
$C_g=M^0\oplus k$ is
\[
  z_u(x,\lambda)=(\lambda g(u),0)=0.
\]
Thus $0\neq z_u\in\operatorname{Ann}_\Gamma(T)$, proving that $T$ is not
faithful.

It remains to verify the permanence assertions.  Assume first that $B$ is
connected.  Then $U\neq0$, because $g$ has nonzero kernel.  Let
\[
  z=
  \begin{pmatrix}
    b&x\\
    0&c
  \end{pmatrix}
\]
be a central idempotent of $\Gamma$.  Commuting with $f$ gives $x=0$, and
commuting with the copy of $B$ in $e\Gamma e$ shows that $b$ is a central
idempotent of $B$.  Thus $b=0$ or $b=1_B$, while $c=0$ or $c=1$ because
$c^2=c$.  Commuting with every element
$\left(\begin{smallmatrix}0&u\\0&0\end{smallmatrix}\right)$ gives
$bu=uc$ for all $u\in U$.  Since $U\neq0$, this forces $b=c1_B$, so
$z=0$ or $z=1_\Gamma$.  Hence $\Gamma$ is connected.

Finally, assume that $B$ is quasi-hereditary.  Since $f\Gamma=kf$, the ideal
\[
  J=\Gamma f\Gamma=\Gamma f
\]
is a projective left $\Gamma$-module.  Moreover,
\[
  J^2=J,
  \qquad
  f\rad\Gamma=0,
  \qquad
  J(\rad\Gamma)J=0,
  \qquad
  \Gamma/J\cong B.
\]
Thus $J$ is a heredity ideal.  To make the last step explicit, let
\[
  0=I_0\subset I_1\subset\cdots\subset I_m=B
\]
be a heredity chain for $B$, and let $\pi\colon \Gamma\to\Gamma/J\cong B$ be the
quotient map.  Then
\[
  0\subset J\subset\pi^{-1}(I_1)\subset\cdots
  \subset\pi^{-1}(I_m)=\Gamma
\]
is a heredity chain for $\Gamma$: the first factor is the heredity ideal
$J$, and every subsequent heredity factor is identified, through $\pi$, with
the corresponding heredity factor $I_i/I_{i-1}$ of $B$.  Therefore
$\Gamma$ is quasi-hereditary.
\end{proof}

\begin{remark}
Recently, Liu, Feng, and Gao \cite{LiuFengGao2026} studied projectively
Wakamatsu tilting modules over one-point extensions, with emphasis on
lifting such modules and their mutations. Our use of one-point extensions
is complementary. We prove a converse reduction after allowing the algebra
to change: every counterexample to the Self-orthogonal Wakamatsu-tilting
Conjecture produces, by a one-point extension, a counterexample to the
Self-orthogonal Faithful Conjecture. 
\end{remark}

We now apply Theorem~\ref{thm:SWC-to-SFC-reduction} to the full-rank
counterexample obtained in Sections~4 and~5. This gives the announced
counterexample to the Self-orthogonal Faithful Conjecture.

\begin{proof}[Proof of Theorem~\ref{thm:SFC-counterexample}]
Let $B$ and $M$ be as in Theorem~\ref{thm:Cn}.  By
Corollary~\ref{cor:enomoto-consequences}, the module $M$ is not Wakamatsu
tilting.  Apply Theorem~\ref{thm:SWC-to-SFC-reduction}.  Since $B$ is basic,
connected, and quasi-hereditary over $\C$, we obtain a basic connected
quasi-hereditary $\C$-algebra $\Gamma$ and a basic self-orthogonal
$\Gamma$-module $T$ such that
\[
  |T|=n(\Gamma)
  \qquad\text{and}\qquad
  \operatorname{Ann}_\Gamma(T)\neq0.
\]
The quasi-heredity of $\Gamma$ implies $\operatorname{gl.dim}\Gamma<\infty$,
so $\pd_\Gamma T<\infty$.  Thus $T$ is pretilting.

A tilting module is faithful: indeed, the defining monomorphism
$\Gamma\to T_0$ with $T_0\in\add(T)$ forces
$\operatorname{Ann}_\Gamma(T)=0$.  Therefore $T$ is not tilting.  Suppose,
more generally, that $T$ were a direct summand of a tilting module $T'$.
After replacing $T'$ by its basic representative, every indecomposable
summand of $T$ still occurs in $T'$.  Since both $T$ and a basic tilting
$\Gamma$-module have exactly $n(\Gamma)$ pairwise non-isomorphic
indecomposable summands, one obtains $\add(T')=\add(T)$.  The tilting
coresolution for $T'$ would then be a tilting coresolution by objects in
$\add(T)$, forcing $T$ itself to be tilting, a contradiction.  Hence $T$ has
no tilting completion.  The asserted counterexample to the Self-orthogonal
Faithful Conjecture follows.
\end{proof}

The preceding reduction has a consequence beyond the particular
counterexample above. Together with the implication observed by
Chen, Li, Zhang, and Zhao, it shows that the two self-orthogonality conjectures are
in fact equivalent when formulated over the class of all
finite-dimensional algebras.

\begin{corollary}
\label{cor:SWC-SFC-equivalent-universally}
As universal assertions over all finite-dimensional algebras, the
Self-orthogonal Wakamatsu-tilting Conjecture and the Self-orthogonal Faithful
Conjecture are equivalent.
\end{corollary}

\begin{proof}
The former conjecture implies the latter because every Wakamatsu tilting
module is faithful; see \cite[Section~5]{ChenLiZhangZhao2025}.  Conversely,
if the former conjecture fails for some algebra $B$, replace the offending
module by its basic representative.  Self-orthogonality, the number of
isomorphism classes of indecomposable summands, and the Wakamatsu-tilting
property depend only on the additive closure of the module, so this remains
a counterexample.  Theorem~\ref{thm:SWC-to-SFC-reduction} then produces a
one-point extension $\Gamma$ for which the latter conjecture fails.  This
proves the converse by contraposition.
\end{proof}

\begin{remark}
Corollary~\ref{cor:SWC-SFC-equivalent-universally} is a statement about the
two conjectures over the class of all finite-dimensional algebras.  It does
not assert that the two properties are equivalent for a fixed algebra: the
reverse reduction replaces $B$ by the one-point extension $\Gamma$.
\end{remark}

\section{The almost-full-rank counterexample}

In this section we pass from the full-rank counterexample of
Theorem~\ref{thm:Cn} to the adjacent almost-full-rank case.  The idea is to
embed the algebra $B$ into a regular one-point extension $\Lambda$ and view
the module $M$ as a $\Lambda$-module.  This increases the rank of the algebra
by one while leaving the number of indecomposable summands of the module
unchanged.  The main point is then to show that any tilting completion over
$\Lambda$ would restrict to a tilting completion over $B$, contradicting
Theorem~\ref{thm:Cn}.  We begin with the standard generation criterion for
generalized tilting modules that will be used in this restriction argument.

\begin{lemma}
\label{lem:tilting-generation}
Let $A$ be a finite-dimensional algebra and let $T$ be a pretilting
$A$-module. Then $T$ is tilting if and only if
\begin{equation}
\label{eq:tilting-generation}
  A\in\thick_{\Db(A\text{-}\mathrm{mod})}(T).
\end{equation}
\end{lemma}

\begin{proof}
If $T$ is tilting, its defining coresolution immediately implies
\eqref{eq:tilting-generation}.

Conversely, since $T$ has finite projective dimension, it is a perfect
complex. Moreover, self-orthogonality gives
$\Hom_{\Db(A\text{-}\mathrm{mod})}(T,T[i])=0$ for every $i>0$, while the
same vanishing holds for $i<0$ since $T$ is a module concentrated in degree
zero. Thus a projective resolution of $T$ is self-orthogonal in
$\Kb(\proj A)$.

Since $T$ is perfect,
$\thick_{\Db(A\text{-}\mathrm{mod})}(T)\subseteq\Kb(\proj A)$.
On the other hand, condition \eqref{eq:tilting-generation} implies that
$A\in\thick(T)$, and hence
$\Kb(\proj A)=\thick(A)\subseteq\thick(T)$.
Therefore $\thick(T)=\Kb(\proj A)$, so a projective resolution of $T$ is a tilting complex. By the standard
characterization of tilting modules in terms of their projective
resolutions, $T$ is a tilting module; see
\cite[Remark~3.3, Theorem~3.5, and Corollary~3.7]{Wei2013}.
\end{proof}

We next recall the restriction functor for a one-point extension.
Restriction and extension of classical tilting modules along one-point
extensions by projective modules were studied by Assem, Happel and Trepode
\cite{AssemHappelTrepode2007}. For the generalized tilting modules considered
in this paper, we will use the homological properties of the restriction
functor recorded in
\cite[Section~2]{AsadollahiPadashnikSadeghiTreffinger2024}.

\begin{lemma}
\label{lem:regular-one-point-restriction}
Let $B$ be a finite-dimensional $\C$-algebra and set
\[
  \Lambda=
  \begin{pmatrix}
    B & B\\
    0 & \C
  \end{pmatrix},
  \qquad
  e=
  \begin{pmatrix}
    1_B & 0\\
    0 & 0
  \end{pmatrix}.
\]
If $T$ is a pretilting $\Lambda$-module, then $eT$ is a pretilting $B=e\Lambda e$-module.
\end{lemma}

\begin{proof}
The algebra $\Lambda$ is the one-point extension of $B$ by the projective
$B$-module $B$. Consider the restriction functor
\[
\mathcal R=\Hom_{\Lambda}(\Lambda e,-):
\Lambda\text{-}\mathrm{mod}
\longrightarrow
B\text{-}\mathrm{mod}.
\]
Evaluation at $e$ gives a natural isomorphism
$\mathcal R(X)\cong eX$ for every $\Lambda$-module $X$.

The functor $\mathcal R$ is exact and preserves projective modules; see
\cite[Section~2]{AsadollahiPadashnikSadeghiTreffinger2024}. Therefore
$\pd_B(eT)\leq\pd_{\Lambda}T<\infty$.

Moreover, by
\cite[Proposition~2.3(3), (4)]{AsadollahiPadashnikSadeghiTreffinger2024},
for any two $\Lambda$-modules $X$ and $Y$ there is a natural epimorphism
$\Ext_{\Lambda}^1(X,Y)\twoheadrightarrow\Ext_B^1(eX,eY)$, and, for every
$r\geq2$, there is a natural isomorphism
$\Ext_{\Lambda}^r(X,Y)\cong\Ext_B^r(eX,eY)$.
Applying these statements to $X=Y=T$ and using
$\Ext_{\Lambda}^r(T,T)=0$ for $r>0$, we obtain
$\Ext_B^r(eT,eT)=0$ for $r>0$.
Hence $eT$ is a pretilting $B$-module.
\end{proof}

The preceding lemma, together with the generation criterion, gives the
generalized version of the usual restriction result for tilting modules
that we need.
\begin{corollary}
\label{cor:regular-one-point-tilting-restriction}
With the notation of Lemma~\ref{lem:regular-one-point-restriction}, if
$T$ is a tilting $\Lambda$-module, then $eT$ is a tilting $B$-module.
\end{corollary}

\begin{proof}
By Lemma~\ref{lem:regular-one-point-restriction}, the $B$-module $eT$ is
pretilting. Let
\[
0\longrightarrow\Lambda\longrightarrow
T_0\longrightarrow T_1\longrightarrow\cdots
\longrightarrow T_s\longrightarrow0,
\qquad
T_i\in\add(T),
\]
be a tilting coresolution of $\Lambda$.

Applying the exact restriction functor $e(-)$ gives an exact sequence
\[
0\longrightarrow e\Lambda\longrightarrow
eT_0\longrightarrow eT_1\longrightarrow\cdots
\longrightarrow eT_s\longrightarrow0,
\qquad
eT_i\in\add(eT).
\]
As a $B$-module,
$e\Lambda=e\Lambda e\oplus e\Lambda f\cong B\oplus B$, where
\[
f=
\begin{pmatrix}
0&0\\
0&1
\end{pmatrix}.
\]
Consequently,
$B\in\thick_{\Db(B\text{-}\mathrm{mod})}(eT)$.
Lemma~\ref{lem:tilting-generation} now shows that $eT$ is a tilting $B$-module.
\end{proof}

\begin{remark}
The classical case of the preceding corollary is
\cite[Corollary~4.3]{AsadollahiPadashnikSadeghiTreffinger2024}; see also
\cite[Proposition~4.1]{AssemHappelTrepode2007}.  Using the higher extension
comparison above, the same restriction property holds for the generalized
tilting modules considered here.
\end{remark}

\begin{proof}[Proof of Theorem~\ref{thm:Cnminusone}]
Let $B$ and $M$ be as in Theorem~\ref{thm:Cn}, and put
\[
  \Lambda=
  \begin{pmatrix}
    B & B\\
    0 & \C
  \end{pmatrix},
  \qquad
  L=\iota(M)=(M,0,0),
\]
where
\[
\iota:B\text{-}\mathrm{mod}
\longrightarrow
\Lambda\text{-}\mathrm{mod},
\qquad
X\longmapsto(X,0,0).
\]

We first record the elementary properties of $\Lambda$. Its radical is
\[
  \rad\Lambda=
  \begin{pmatrix}
    \rad B & B\\
    0 & 0
  \end{pmatrix},
\]
and hence
$\Lambda/\rad\Lambda\cong(B/\rad B)\times\C$.
Since $B$ is basic, $\Lambda$ is basic and
$n(\Lambda)=n(B)+1$.

We next prove that $\Lambda$ is connected. Let
\[
  z=
  \begin{pmatrix}
    b & x\\
    0 & c
  \end{pmatrix}
\]
be a central idempotent of $\Lambda$. Commuting $z$ with
\[
f=
\begin{pmatrix}
0&0\\
0&1
\end{pmatrix}
\]
gives $x=0$. For every $y\in B$, commuting $z$ with
\[
  \begin{pmatrix}
    0 & y\\
    0 & 0
  \end{pmatrix}
\]
gives $by=yc$. Taking $y=1_B$, we obtain $b=c1_B$. Since
$c\in\C$ is idempotent, $c=0$ or $c=1$, and hence $z=0$ or
$z=1_{\Lambda}$. Thus $\Lambda$ is connected.

Moreover, $f\Lambda=\C f$, so $I=\Lambda f\Lambda=\Lambda f$.
In particular, $I$ is a projective left ideal. It satisfies
$I^2=I$ and $I(\rad\Lambda)I=0$, and
$\Lambda/I\cong B$.
Thus $I$ is a heredity ideal. Since $B$ is quasi-hereditary, $\Lambda$ is
quasi-hereditary.

We now show that $L$ is pretilting. The functor $\iota$ is exact, fully
faithful, and sends projective $B$-modules to projective $\Lambda$-modules.
Thus \(\pd_{\Lambda}L\leq\pd_BM\).
Conversely, the exact restriction functor $e(-)$ preserves projectives, so
applying it to a projective $\Lambda$-resolution of $L$ gives a projective
$B$-resolution of $M=eL$. Hence \(\pd_BM\leq\pd_{\Lambda}L\),
and consequently \(\pd_{\Lambda}L=\pd_BM<\infty\).

If $P_{\bullet}\to M$ is a projective $B$-resolution, then
$\iota(P_{\bullet})\to L$ is a projective $\Lambda$-resolution. Since
$\iota$ is fully faithful, for every $r\geq0$ we have \(\Ext_{\Lambda}^r(L,L) \cong \Ext_B^r(M,M)\). It follows that $\Ext_{\Lambda}^r(L,L)=0$ for $r>0$.
Thus $L$ is pretilting.

The functor $\iota$ also preserves indecomposability and isomorphism classes. Indeed, \(\operatorname{End}_{\Lambda}(\iota X) \cong \operatorname{End}_B(X)\), and an isomorphism $\iota X\cong\iota Y$ induces an isomorphism
$X\cong Y$. Therefore
\[
  |L|
  =
  |M|
  =
  n(B)
  =
  n(\Lambda)-1.
\]
Hence $L$ is a basic almost-tilting $\Lambda$-module.

Suppose that $L$ were a direct summand of a tilting $\Lambda$-module $T$.
By Corollary~\ref{cor:regular-one-point-tilting-restriction}, $eT$ is a
tilting $B$-module. Since $L$ is a direct summand of $T$, the module $M=eL$ is a direct summand of $eT$. This contradicts Theorem~\ref{thm:Cn}.
Therefore $L$ has no tilting completion.
\end{proof}

\begin{remark}
The module $L$ in Theorem~\ref{thm:Cnminusone} is not faithful.  Since
$M$ is faithful over $B$, the annihilator of $L=\iota(M)$ is precisely
$\Lambda f\Lambda$.  Thus Theorem~\ref{thm:Cnminusone} settles
$\mathrm{(C_{n-1})}$ for connected quasi-hereditary algebras, but it does
not provide a faithful almost-full-rank counterexample.
\end{remark}

\begin{remark}
Let $r=n(B)$ for the algebra in Theorem~\ref{thm:Cn}. For every $q\geq0$, replace $(B,M)$ by $B_q=B\times\C^q$ and $M_q=M\oplus\C^q$. For $q=0$, the assertions below are exactly those of
Theorem~\ref{thm:Cn}. Assume that $q>0$. Since finite direct products
of quasi-hereditary algebras are quasi-hereditary, $B_q$ is
quasi-hereditary. Under the canonical equivalence
\[
  B_q\text{-}\mathrm{mod}
  \simeq
  B\text{-}\mathrm{mod}
  \times
  \C^q\text{-}\mathrm{mod},
\]
projective dimension, extension groups, and the tilting coresolution
condition are all computed componentwise. Since \(n(B_q)=n(B)+q\), the module $M_q$ is a basic full-rank pretilting $B_q$-module.

Moreover, if $M_q$ were a direct summand of a tilting $B_q$-module
$T$, then the projection of $T$ to $B\text{-}\mathrm{mod}$ would be a
tilting $B$-module containing $M$ as a direct summand, contradicting
Theorem~\ref{thm:Cn}. Hence $M_q$ has no tilting completion.

Applying the regular one-point extension above produces a connected
quasi-hereditary algebra of rank $r+q+1$ with a non-completable pretilting
module of rank $r+q$. Consequently, connected counterexamples to
$\mathrm{(C_{n-1})}$ exist for every sufficiently large rank $n$.
\end{remark}

\section*{Acknowledgments}
The first author is supported by the Fundamental Research Funds for the
Central Universities (No.~GK202403003) and the National Natural Science
Foundation of China (Grant No.~12271321).
    
\section*{Statement on AI usage}

We acknowledge the use of AI during the preparation of this manuscript. The overall strategy of the paper was developed by the authors.
ChatGPT assisted in simplifying the proof of Lemma \ref{lem:one-point extension-ext-calculus}, and in verifying that the one-point extension $\Gamma$ in the proof of Theorem \ref{thm:SWC-to-SFC-reduction} is still quasi-hereditary. ChatGPT brought to our attention the related work \cite{LiuFengGao2026}. It was also used for language editing. The resulting arguments were subsequently checked in detail, revised, and integrated into the manuscript by the authors.

The authors have reviewed the manuscript in full and take full responsibility for all mathematical claims, arguments, and conclusions, as well as for any remaining errors.

\end{document}